\documentclass[11pt, letterpaper]{amsart}
\usepackage{graphicx, amssymb, xy, bm}

\newtheorem{theorem}{Theorem}[section]

\newtheorem{corollary}[theorem]{Corollary}

\newtheorem{definition}[theorem]{Definition}

\newtheorem{lemma}[theorem]{Lemma}

\numberwithin{theorem}{section}

\newcommand{\mc}{\mathcal}

\newcommand{\ra}{\rightarrow}
\newcommand{\mb}{\mathbb}

\newcommand{\F}{\mathcal{F}}

\title[]{Quickest Search over Brownian Channels}

\author[]{Erhan Bayraktar}\thanks{This work is supported by the National Science Foundation under grant DMS-1118673.  We would like to thank the referees, associate editor, and editor for their helpful comments, which have greatly improved the paper} 
\address[Erhan Bayraktar]{Department of Mathematics, University of Michigan, 530 Church Street, Ann Arbor, MI 48104, USA}
\email{erhan@umich.edu}
\author[]{Ross Kravitz}
\address[Ross Kravitz]{Department of Mathematics, University of Michigan, 530 Church Street, Ann Arbor, MI 48104, USA}
\email{ross.kravitz@gmail.com}

\begin{document}
\begin{abstract} In this paper we resolve an open problem proposed by \cite{MR2849122}.  Consider a sequence of Brownian Motions with unknown drift equal to one or zero, which may be observed one at a time.  We give a procedure for finding, as quickly as possible, a process which is a Brownian Motion with nonzero drift.  This original quickest search problem, in which the filtration itself is dependent on the observation strategy, is reduced to a single filtration impulse control and optimal stopping problem, which is in turn reduced to an optimal stopping problem for a \emph{reflected} diffusion, which can be \emph{explicitly} solved.
\end{abstract}
\keywords{Bayesian quickest search, optimal switching, optimal stopping, reflected diffusion.}
\maketitle

\section{Introduction}

In the classical sequential analysis problem, one observes an i.i.d. sequence of random variables $\{X_1,X_2,\ldots\}$ which all either have distribution $Q_0$ or $Q_1$.  In deciding whether $Q_0$ or $Q_1$ is present, the agent seeks to minimize some combination of the probability of choosing the wrong distribution and the total expected observation time.  The first results in this field are due to Wald in \cite{MR0013275}, who solves the above problem.  The problem was first addressed in continuous time by Shiryaev in \cite{MR0272119}.  In one continuous version of the sequential analysis problem, the agent observes a Brownian Motion with an unknown drift, which satisfies one of two hypotheses.  He seeks to minimize some combination of the probability of misidentifying the drift and the total expected observation time.  For a comprehensive introduction to the topic, the reader is referred to \cite{MR2482527}.  \cite{MR2256030} also has an excellent chapter on the subject.

Since their original formulation, many variants of these sequential analysis problems have been studied.  The problem that we study is known as a quickest search problem.  In this class of problems, the single channel of information available in the classical problem is replaced with some collection of channels, which may be finite or infinite.  By observing only one channel at a time, the goal is to find a channel satisfying some hypothesis, all the while keeping observation costs to a minimum.  This type of formulation is a natural one for many physical applications: due to hardware or bandwith constraints, our ability to monitor a system may often be constrained to a single part at a time.  We refer to \cite{MR2849122} for a further description of possible applications.  

In \cite{MR2849122}, the following problem is studied: consider countably many sequences $\{Y^i_k : k = 1,2,\ldots \}$, $i = 1,\ldots$.  For each $i$, $\{Y^i_k : k = 1, 2, \ldots \}$ are i.i.d. random variables which obey one of two hypotheses: under $H_0$, $Y^i_k \sim Q_0, k = 1,2,\ldots$, and under $H_1$, $Y^i_k \sim Q_1, k = 1,2,\ldots$, where $Q_0$ and $Q_1$ are two distinct, but equivalent, distributions.  At each discrete time $k$, one can take one of three actions: stop sampling and choose a sequence which is believed to satisfy $H_1$, continue observation of the same channel, or continue observation in a new channel.  Over all possible observation strategies and their associated stopping times $\tau$, our goal is to minimize $P(H^{s_\tau} = H_0) + cE [\tau]$, where here $H^{s_\tau}$ is the true condition of the channel observed at time $\tau$, and $c$ represents the cost of making one observation.

The authors of \cite{MR2849122} solve the problem above, in the sense that they find an optimal observation strategy and stopping time, both of which can be computed as hitting times of an underlying posterior process.  In this same paper, the authors ask for a solution to the corresponding problem in continuous time, and it is this task which we take up.  In the quickest search problem in continuous time, the basic object of study is a sequence of observable processes $\xi^i_t = \theta^i t + B^i_t$, $i \in \mb{N}$, where each $B^i$ is an independent Brownian Motion and each $\theta^i$ is an independent Bernoulli random variable.  For some prior $\hat{\pi} \in (0,1)$ and each $i$, $P_{\hat{\pi}}(\theta^i = 0) = \hat{\pi}$, and $P_{\hat{\pi}}(\theta^i = 1) = 1 - \hat{\pi}$.  We say that $\xi^i$ satisfies hypothesis $H_0$ if $\theta^i = 0$ and $\xi^i$ satisfies hypothesis $H_1$ if $\theta^i = 1$.  The general objective in the quickest search problem is to find, as quickly as possible, a process $\xi^i$ which satisfies $H_1$.  Observing any process for $t$ units of time incurs a cost of $ct$, where $c>0$ is a constant.  One may observe only one $\xi^i$ at a time, but one can instantaneously change between observed processes at any time. In discrete time, the problem may be approximated by a finite horizon version solvable by backwards induction, but this method is not available in continuous time.  Instead, the problem is solved by formulating it as a free boundary ordinary differential equation.

In the literature, there has been a large amount of research into quickest search problems, although the majority of it has been in discrete time.  Since the focus of our paper is a continuous time problem, we refer the reader to the references in \cite{MR2849122} for an excellent description of the discrete-time literature, as well as \cite{MR2790111} for a multi-channel quickest detection problem in discrete time.  The continuous time literature is sparser, but there are several papers addressing a problem similar to ours.  The three papers \cite{MR0200090}, \cite{Posner:2006:CSD:2263257.2267360}, and \cite{MR0345225} all consider the quickest search variant in which there are a finite number of channels, and exactly one channel satisfies $H_1$.  For comparison, in the problem we study, there are infinitely many channels and no knowledge of how many channels satisfy $H_1$.  Also in continuous time, in \cite{MR2384820}, the authors study a multi-channel problem where all channels are observed simultaneously: here the goal is to find the common intensity rate of a collection of Poisson processes.  In \cite{MR2299631}, two Poisson channels are observed simultaneously to determine the one in which disruption occurs first.  Other problems involving multiple stopping times include \cite{MR2395575} and \cite{MR2416002}, although our problem is closer in spirit to a multiple switching problem than one of multiple stopping.

The technical details of continuous time formulations are somewhat subtle, and indeed, in the three papers closest to ours, \cite{MR0200090}, \cite{Posner:2006:CSD:2263257.2267360}, and \cite{MR0345225} , each purports to fix an error in the previous one.  For example, in \cite{MR0200090}, an optimal switching strategy is described in the following way: consider $N$ diffusions $\gamma^n_t$, $1 \leq n \leq N$, which are supposed to represent the posterior for each channel, and take the strategy that when $\gamma^i_t$ is largest among the $N$ processes, observe channel $i$.  The inherent problem in such a strategy is that when, for example, $\gamma^i_t = \gamma^j_t$, the sign of $\gamma^i_s - \gamma^j_s$ will oscillate infinitely many times when $s$ is in a neighborhood of $t$.  So, such a strategy would necessarily switch between channels $i$ and $j$ infinitely many times in that neighborhod, and this is physically unfeasible.

Indeed, even in the discrete time case, certain technical details are not completely developed.  To rigorously describe the set of all observation strategies in our problem, one must talk of different filtrations, since by choosing which channel to observe at different times, we are modulating the exact information which we receive.  Therefore, one of the aims of this paper is describe a continuous time quickest search problem in a mathematically rigorous way.  At the same time, however, we can obtain a closed form analytic solution to our quickest search problem, both for the value function and the optimal stopping time, and so from a practical perspective, our results may be useful in the discrete time case if calculating the solution to that problem is too expensive.  The theory of extended weak convergence, in \cite{aldous1981weak}, provides a possible way to translate insights from continuous time optimal stopping back to discrete time.

The outline of this paper is as follows: we will first describe how problems in continuous time sequential analysis can be formulated as optimal stopping problems, using the building block of the one channel problem.  Much of this material is derived from parts of \cite{MR2256030}, Chapter $6$, and \cite{MR2374974}, Chapter $4$.  We will then state an optimal stopping/switching problem which models our quickest detection problem.  Its basic feature is that it involves many different filtrations, corresponding to the different ways in which one may observe the processes.  We show how this problem may be reduced to an impulse control/stopping problem with a single filtration and a single Brownian Motion, so that the tools of optimal stopping may be applied.  

	Let us describe this impulse control/stopping problem more closely, since here the exact structure of the problem is evident.  Consider a sequence of channels which all have prior probability $\hat{\pi}$ of satisfying $H_1$.  By observing one channel at the time, the posterior process of precisely one of these channels will evolve in time according to a stochastic differential equation, and all the rest will simply stay at $\hat{\pi}$.  The goal of the agent is essentially to find some channel and some time when the respective posterior process is very close to $1$, ensuring that $H_1$ is almost certainly satisfied.  Given this goal, the agent, when faced with a posterior which has dropped below $\hat{\pi}$, will want to immediately move on to the next channel.  The agent must take a finite time to react, so after some amount of time $\epsilon$, or more realistically, when the posterior process hits the level $\hat{\pi} - \epsilon$, he switches channels, and the effect of this is that the posterior process is reset to $\hat{\pi}$.  These are the impulse controls available to the agent: at any time, he can move ``his" posterior process back to $\hat{\pi}$.  Now, in the limit, the agent ideally wants to react as quickly as possible, which means $\epsilon \downarrow 0$.  In this limit, we show that these impulsed processes converge in an appropriate sense to a diffusion with reflecting boundary at $\hat{\pi}$, so that the original problem may be stated as an optimal stopping problem on this reflected process.

	Finally, we solve the optimal stopping problem using a standard verification theorem.  Using the results derived thus far, we outline $\epsilon$-optimal algorithms for the quickest search problem, and provide some computations of optimal threshold levels for different parameters.


\section{From Sequential Analysis to Optimal Stopping}

In this section we will describe how problems in continuous time sequential analysis can be formulated as optimal stopping problems, using the building block of the one channel problem.  This material is derived from Chapter $4$ of \cite{MR2374974}. 

Let $\left(\Omega, \mc{F}, P_{\hat{\pi}}\right)$ be a probability space supporting a Brownian Motion $B$ and an independent Bernoulli random variable $\theta$, with $P_{\hat{\pi}}(\theta = 1) = \hat{\pi}$.  We consider two statistical hypotheses

\[
H_1: \theta = 1 \mbox{ and } H_0: \theta = 0,
\]

\noindent which have respective prior probabilities $\hat{\pi}$ and $1 - \hat{\pi}$.  We let $\xi_t = \theta t + B_t$ model the observed process, with induced filtration $\mb{F}^\xi = \left(\mc{F}^\xi_t \right)_{t \geq 0}$.  We will find it useful to write $P_{\hat{\pi}} = \hat{\pi} P_1 + (1-\hat{\pi})P_0$, where under $P_1$, $\xi_t$ is a Brownian Motion with unit drift, and under $P_0$ it is a standard Brownian Motion with zero drift.

Based on the continuous observation of $\xi$, our goal is to choose a sequential decision rule $(\tau, d)$, where $\tau$ is a $\mb{F}^\xi$ stopping time, and $d$ is $\mc{F}^\xi_\tau$-measurable and equal to one or zero.  Taking $d=1$ models accepting $H_1$ at time $\tau$, and taking $d = 0$ models accepting $H_0$.  The goal is to minimize the risk function
\begin{equation}\label{introeq}
V(\pi) = \underset{(\tau, d)}{\inf} \ E_\pi \left[ \tau + a 1_{\{d = 0, \theta = 1\}} + b1_{\{d = 1, \theta = 0\}} \right],
\end{equation}
where $a$ and $b$ are used to weight the importance of each misidentification.  Let $\pi_t = P_{\hat{\pi}}(\theta = 1 | \F^\xi_t)$ be the posterior process, which models our belief that $H_1$ is satisfied, based on observations up to time $t$, $\mc{F}^\xi_t$.  Given a stopping time $\tau$, the choice of whether to set $d = 0$ or $d = 1$ is completely described by the value of $\pi_\tau$.  For $c = \frac{b}{a + b}$, if $\pi_\tau \leq c$, $d = 0$, and $d = 1$ otherwise.  Therefore, \eqref{introeq} may be restated as an optimal stopping problem on $\pi_t$:
\[
V(\pi) = \underset{\tau}{\inf} \ E_\pi \left[\tau + a\pi_\tau \wedge b(1 - \pi_\tau) \right].
\]

We have reduced the sequential analysis problem to one of optimal stopping, but it remains to understand the time dynamics of the posterior process $\pi_t$.  To that end, we introduce the odds process
\[
\Phi_t \triangleq \frac{dP_1}{dP_0} \Big|_{\mc{F}^\xi_t}(\omega).
\]
By an application of Bayes' rule, $\Phi_t$ can be calculated, and is equal to $\exp \left\{ \xi_t - \frac{t}{2} \right\}$. From \cite{MR2374974}, p. 181, the posterior process $\pi_t$ satisfies
\[
\pi_t = \hat{\pi}\frac{dP_1}{d[\hat{\pi}P_1 + (1 - \hat{\pi}) P_0]} \Big|_{\mc{F}^\xi_t}(\omega),
\]
so we can relate $\pi_t$ to $\Phi_t$.  Following the reasoning on p.181 of \cite{MR2374974}, we may derive




\[
d \pi_t = \pi_t(1 - \pi_t) dW_t, \pi_0 = \hat{\pi}.
\]

The optimal stopping problem on this diffusion will play a central role in our analysis.

\section{Reduction to a problem with a single filtration}

Let $\Omega^i$, $i \in \mb{N}$, denote $C[0,\infty)$.  We consider a sequence of independent Brownian Motions $W^1,W^2,\ldots$, defined canonically as the coordinate process on $\Omega^{\mb{N}} = \Omega^1 \times \Omega^2 \times \cdots$: for $(\omega^1,\omega^2,\ldots) \in \Omega^{\mb{N}}$, $W^i_t(\omega^1,\omega^2,\ldots) = \omega^i_t$.  Let $\mb{F}$ be the filtration generated by the canonical coordinate process.  Let $P^{\mb{N}}$ denote the product measure of Wiener measures on $\Omega^{\mb{N}}$.

We define, for each $i$, and for any random time $\phi$ independent of $W^i$, the process $\pi^{i,\phi}$ satisfying
\begin{equation}\label{pidef}
d\pi^{i,\phi}_t = \pi_t^{i,\phi} (1-\pi^{i,\phi}_t)dW^i_t, \mbox{ for } t \geq \phi, \pi^{i,\phi}_\phi = \hat{\pi}.
\end{equation}

The value $\pi^{i, \phi}_t$ represents the posterior probability, based on observing the history of Channel $i$ from time $\phi$ up until time $t$, that Channel $i$ satisfies $H_1$.  Although $\phi$ here is arbitrary, it will always be, for our purposes, a switching time between two channels.  We will take these stochastic differential equations as our main object of study.  Now, given that the single channel sequential analysis problem involves the optimal stopping of a single copy of the SDE, it stands to reason that a multi channel problem should involve a sequence of these processes which are concatenated together according to the way in which we observe each channel.  We describe this procedure now.

Let $\mb{F}^{(1)}$ be the filtration on $\Omega^{\mb{N}}$ generated by $W^1$, which coincides with the filtration generated by $\pi^1 \triangleq \pi^{1,0}$.  We let $\mc{T}^{(1)}$ be the set of $\mb{F}^{(1)}$-stopping times.    

Let $\mathfrak{S}$ be the set of admissible switching controls, which will determine the set of times when the currently observed channel is changed.  Elements of $\mathfrak{S}$ will consist of sequences of increasing random times $\{\phi_1,\phi_2,\ldots\}$ with $\phi_1 = 0$.  The time $\phi_i$, $i \geq 2$, may be interpreted as the time when observation of Channel $i-1$ stops and observation of Channel $i$ begins.  The main property that each $\phi_i$ should have is that it should be measurable with respect to the information gathered before it: our decision to switch should be based on what we have seen thus far.  In order to make this precise, we will define elements of $\mathfrak{S}$ inductively: given the first $n$ switching times $\phi_1,\phi_2,\ldots,\phi_n$, we will define an allowed $(n+1)^{st}$ switching time.  First, the base case.  We let $\phi_1 = 0$.

The first possible switching time $\phi_2$ is any strictly positive $\mb{F}^{(1)}$-stopping time.  So, given $\phi_2$, we define the process $\pi^{(2),\phi_2}$ as follows:
\[
\pi^{(2),\phi_2}_t = \pi^1_t 1_{\{t < \phi_2\}} + \pi^{2,\phi_2}_t 1_{\{t \geq \phi_2\}}.
\]

The process $\pi^{(2),\phi_2}$ generates a filtration $\mb{F}^{(2),\phi_2}$.  We let $\mc{T}^{(2),\phi_2}$ denote the set of $\mb{F}^{(2),\phi_2}$-stopping times.

Now, we define what the switching time $\phi_3$ may look like, given that $\phi_2$ has already been chosen.  Such a switching time is any $\phi_3 \in \mc{T}^{(2),\phi_2}$ such that $\phi_3 > \phi_2$.  Given $\phi_2$ and $\phi_3$, we define the process $\pi^{(3),\phi_2,\phi_3}$ as follows:
\[
\pi^{(3),\phi_2,\phi_3}_t = \pi^{(2),\phi_2}_t 1_{\{t < \phi_3\}} + \pi^{3,\phi_3}_t 1_{\{t \geq \phi_3\}}.
\]

The process $\pi^{(3),\phi_2,\phi_3}$ generates a filtration $\mb{F}^{(3),\phi_2,\phi_3}$, and $\mc{T}^{(3),\phi_2,\phi_3}$ denotes the set of $\mb{F}^{(3),\phi_2,\phi_3}$-stopping times.  Proceeding in this way, we define, for each $n \in \mb{N}$, $\pi^{(n),\phi_2,\ldots,\phi_n}, \mb{F}^{(n),\phi_2,\ldots,\phi_n}$, and $\mc{T}^{(n),\phi_2,\ldots,\phi_n}$.  These are, respectively, the posterior process, filtration, and stopping times which result from switching channels at times $\phi_2,\phi_3,\ldots,\phi_n$.

\begin{definition}\label{def1}  Let $\Phi = \{\phi_1,\phi_2,\ldots\}$ be a sequence of random times such that $\phi_i > \phi_{i-1}$ on the set $\{\phi_{i-1} < \infty\}$, and such that $\lim_i \phi_i = \infty$.  We say that $\Phi$ is an admissible switching control if $\phi_1 = 0$, $\phi_2 \in \mc{T}^{(1)}$, and for $n \geq 2$, $\phi_n \in \mc{T}^{(n-1),\phi_2,\ldots,\phi_{n-1}}$
\end{definition}

We denote by $\mathfrak{S}$ the set of all possible switching strategies $\Phi$.  Each $\Phi = \{\phi_1,\phi_2,\ldots\} \in \mathfrak{S}$ induces an observed posterior process $\pi^\Phi$, defined as follows:
\begin{equation}\label{piphidef}
\pi^\Phi_t = \pi^{(n),\phi_2,\ldots,\phi_n}_t \text{ on the set } \{\phi_n \leq t < \phi_{n+1}\}.
\end{equation}
Intuitively, in comparison with \eqref{pidef}, $\pi^\Phi_t$ represents the posterior probability that at time $t$, the channel currently being observed under the observation strategy $\Phi$ satisfies hypothesis $H_1$.  If the same channel was always observed, $\pi^\Phi$ would behave exactly like $\pi^{1,\phi = 0}$.  As it is, when the channel is switched, the effect on the posterior is a sudden jump back to the original level $\hat{\pi}$.

The process $\pi^\Phi$ induces the filtration $\mb{F}^\Phi$ along with $\mc{T}^\Phi$, the set of $\mb{F}^\Phi$-stopping times.  We define the value function as follows:
\begin{equation}\label{value}
V_{\hat{\pi}} \triangleq \underset{\Phi \in \mathfrak{S}}{\inf} \  V_{\hat{\pi}}^\Phi \triangleq \underset{\Phi \in \mathfrak{S}}{\inf} \ \underset{\tau \in \mc{T}^\Phi}{\inf} E[c \tau + (1-\pi^\Phi_\tau)].
\end{equation}

Note that in the value function \eqref{piphidef}, there is only a $(1 - \pi_t)$ term, instead of $\pi_t \wedge (1 - \pi_t)$.  This reflects the fact that instead of deciding whether a single channel satisfies $H_1$ or $H_0$, one is looking only for a channel which satisfies $H_1$.  Now, the $\Phi$ also induces a process $W^\Phi$, which is defined as follows on $\Omega^{\mb{N}}$:

\begin{equation}\label{wphi}
W^\Phi_t(\omega^1,\omega^2,\ldots) \triangleq \omega^1_t \text{ on the set } \{t < \phi_2\}, 
\end{equation}
and for $n \in \mb{N}$
\begin{equation}\label{wphi2}
W^\Phi_t(\omega^1,\omega^2,\ldots) \triangleq W^\Phi_{\phi_n-}(\omega^1,\omega^2,\ldots) + (\omega^n_t - \omega^n_{\phi_n}) \text{ on the set } \{\phi_n \leq t < \phi_{n+1}\}.
\end{equation}

We prove the following standard fact:

\begin{lemma}\label{lemma1}  For each $\Phi \in \mathfrak{S}$, $W^\Phi$ is a Brownian Motion.
\end{lemma}

\begin{proof} We first prove that $W^\Phi$ is a martingale.  Note that $W^\Phi_t$ can be written as
\[
W^\Phi_t = \sum_{i=1}^\infty (W^i_{t \wedge \phi_{i+1}} - W^i_{t \wedge \phi_i}).
\]
Additionally, 
\[
E[(W^\Phi_t)^2] = \sum_{i=1}^\infty E[(W^i_{t \wedge \phi_{i+1}} - W^i_{t \wedge \phi_i})^2] = \sum_{i=1}^\infty (E[t \wedge \phi_{i+1}] - E[t \wedge \phi_i]) = t,
\]
since $\lim_{\phi_i} = \infty$.  Therefore, by the Dominated Convergence Theorem,
\[
\begin{split}
E[W^\Phi_t | \mc{F}_s] 
&= E \left[ \sum_{i=1}^\infty (W^i_{t \wedge \phi_{i+1}} - W^i_{t \wedge \phi_i}) | \mc{F}_s \right]
\\&= \sum_{i=1}^\infty (W^i_{s \wedge \phi_{i+1}} - W^i_{s \wedge \phi_i})
\\& = W^\Phi_s,
\end{split}
\]
the second inequality following from Optional Sampling and the fact that each $W^i$ is a martingale.  Since $\langle W^i \rangle_t = t$ for all $t$, a.s., it follows that $\langle W^\Phi \rangle_t = t$ for all $t$, a.s.  It is also clear by construction that $W^\Phi$ has continuous paths.  Thus by Levy's characterization of Brownian Motion, $W^\Phi$ is a Brownian Motion for each $\Phi \in \mathfrak{S}$.
\end{proof}

The process $\pi^\Phi$ has continuous paths, with the exception of jump times at $\phi_2,\phi_3,\ldots$.  

\begin{lemma}\label{lemma2} $\pi^\Phi_t = \hat{\pi} + \int_0^t \pi^\Phi_s (1-\pi^\Phi_s) dW^\Phi_s + \sum_{i=1}^\infty (\hat{\pi} - \pi^\Phi_{\phi_i-}) 1_{\{t \geq \phi_i\}}$
\end{lemma}

\begin{proof}  By \eqref{piphidef} and \eqref{pidef}, on $\{\phi_n \leq t < \phi_{n+1} \}$, $\pi^\Phi_t = \pi^\Phi_{\phi_n} + \int_{\phi_n}^t \pi^\Phi_s \left(1 - \pi^\Phi_s \right) dW^n_s$.  Furthermore, on $\{\phi_n \leq t < \phi_{n+1}\}$, $W^\Phi_t - W^\Phi_{\phi_n} = W^n_t - W^n_{\phi_n}$ by construction of $W^\Phi$.  Using the locality of stochastic integration (see \cite{MR2273672}, the Corollary on p. $62$), this implies that $\pi^\Phi_t = \pi^\Phi_{\phi_n} + \int_{\phi_n}^t \pi^\Phi_s (1-\pi^\Phi_s) dW^\Phi_s$.  In particular, $\pi^\Phi_{\phi_{n+1}-} = \pi^\Phi_{\phi_n} + \int_{\phi_n}^{\phi_{n+1}} \pi^\Phi_s (1-\pi^\Phi_s) dW^\Phi_s$.  Finally, $\Delta \pi^\Phi_{\phi_n} \triangleq \pi^\Phi_{\phi_n} - \pi^\Phi_{\phi_n-} = \hat{\pi} - \pi^\Phi_{\phi_n-}$.  Then, on $\{\phi_n \leq t < \phi_{n+1}\}$,
\[
\begin{split}
\pi^\Phi_t
&= \pi^\Phi_{\phi_n} + \int_{\phi_n}^t \pi^\Phi_s (1-\pi^\Phi_s) dW^\Phi_s
\\&= \pi^\Phi_{\phi_n-} + \Delta \pi^\Phi_{\phi_n} + \int_{\phi_n}^t \pi^\Phi_s (1-\pi^\Phi_s) dW^\Phi_s
\\&=\pi^\Phi_{\phi_{n-1}} + \int_{\phi_{n-1}}^{\phi_n} \pi^\Phi_s (1-\pi^\Phi_s) dW^\Phi_s + \pi^\Phi_{\phi_n} - \pi^\Phi_{\phi_n-} + \int_{\phi_n}^t \pi^\Phi_s (1-\pi^\Phi_s) dW^\Phi_s
\\&= \pi^\Phi_{\phi_{n-1}} + \int_{\phi_{n-1}}^t \pi^\Phi_s (1-\pi^\Phi_s) dW^\Phi_s + \hat{\pi} - \pi^\Phi_{\phi_n-}
\end{split}
\]
Now, if we repeated apply this procedure, reducing the index $n$ by one each time, we obtain
\[
\pi^\Phi_0 + \int_0^t \pi^\Phi_s (1-\pi^\Phi_s) dW^\Phi_s + \sum_{i=1}^n (\hat{\pi} - \pi^\Phi_{\phi_i-}) = \hat{\pi} + \int_0^t \pi^\Phi_s (1-\pi^\Phi_s) dW^\Phi_s + \sum_{i=1}^n (\hat{\pi} - \pi^\Phi_{\phi_i-}).
\]
\end{proof}

Let $\overline{\Omega}$ be another copy of the canonical space $C[0,\infty)$ with coordinate process $\overline{W}_t$ and filtration $\overline{\mb{F}}$ generated by $\overline{W}$.  Let $\overline{P}$ denote Wiener measure on this space.  Also, let $\overline{\mc{T}}$ denote the set of $\overline{\mb{F}}$-stopping times.  We would like to reduce the original problem $V_{\hat{\pi}}$ to one where everything uses the \emph{same} Brownian Motion ($\overline{W})$ and \emph{same} filtration ($\overline{\mb{F}}$).

\begin{lemma}\label{lemma3}  Let $\Phi \in \mathfrak{S}$.  For any $\mb{F}^\Phi$-stopping time $\tau$, there exists a $\overline{\tau} \in \overline{\mc{T}}$ such that $W^\Phi_{\cdot \wedge \tau}$ and $\overline{W}_{\cdot \wedge \overline{\tau}}$ are identically distributed as processes.  Conversely, for any  $\overline{\tau} \in \overline{\mc{T}}$, there exists a $\mb{F}^\Phi$-stopping time $\tau$ such that $W^\Phi_{\cdot \wedge \tau}$ and $\overline{W}_{\cdot \wedge \overline{\tau}}$ are identically distributed as processes.
\end{lemma}

\begin{proof}  Let $\Phi \in \mathfrak{S}$, and let $\tau \in \mc{T}^\Phi$.  We have a mapping $W^\Phi: \Omega^\mb{N} \ra \overline{\Omega}$ which is defined according to \eqref{wphi} and \eqref{wphi2}.  Since $\tau \in \mc{T}^\Phi$, it is in particular measurable with respect to the filtration $\mb{F}^\Phi$ generated by $W^\Phi$ on $\Omega^{\mb{N}}$.  This implies that for any $\overline{\omega} \in \overline{\Omega}$, $\tau$ is constant on $(W^\Phi)^{-1}(\overline{\omega})$.  

Thus, we define $\overline{\tau}:\overline{\Omega} \ra \mb{R}$ as follows.  For $\overline{\omega} \in \overline{\Omega}$, choose any $\omega \in (W^\Phi)^{-1}(\overline{\omega})$, and set $\overline{\tau}(\overline{\omega}) \triangleq \tau(\omega)$.  From the discussion in the above paragraph, this definition is well-defined.  

We claim that $\overline{\tau}$ is an $\overline{\mb{F}}$-stopping time.  For $t \in \mb{R}$, we have
\[
\begin{split}
\{\overline{\omega} : \overline{\tau}(\overline{\omega}) \leq t \}
&= \{ \overline{\omega} : \tau(\omega) \leq t \text{ for } \omega \in (W^\Phi)^{-1}(\overline{\omega}) \}
\\&= W^\Phi \left( \{ \omega : \tau(\omega) \leq t \} \right).
\end{split}
\]
Since $\tau \in \mc{T}^\Phi$, the set $\{ \omega : \tau(\omega) \leq t \} \in \mc{F}^\Phi_t$.  By construction, the mapping $W^\Phi : \Omega^\mb{N} \ra \overline{\Omega}$ takes $\mc{F}^\Phi_t$-measurable sets into $\overline{\mc{F}}_t$-measurable sets.  Therefore, $\overline{\tau} \in \overline{\mc{T}}$.

Next, we claim that $W^\Phi_{\cdot \wedge \tau}$ and $\overline{W}_{\cdot \wedge \overline{\tau}}$ are distributed identically as processes.  As before, this is essentially a tautology.  Let $A \in \overline{\mc{F}}_\infty$.

We have
\[
\begin{split}
\{\overline{\omega} : \overline{W}_{\cdot \wedge \overline{\tau}(\overline{\omega})}(\overline{\omega}) \in A \}
&= \{ \overline{\omega} : W^\Phi_{\cdot \wedge \tau(\omega)}(\omega) \in A \text{ for } \omega \in (W^\Phi)^{-1}(\overline{\omega}) \}
\\&= W^\Phi \left( \{ \omega : W^\Phi_{\cdot \wedge \tau(\omega)}(\omega) \in A \} \right).
\end{split}
\]

Thus, $\overline{P}(\{\overline{\omega} : \overline{W}_{\cdot \wedge \overline{\tau}(\overline{\omega})} \in A \}) = \overline{P} \left(W^\Phi \left( \{ \omega : W^\Phi_{\cdot \wedge \tau(\omega)}(\omega) \in A \} \right) \right)$.  Since $W^\Phi$ is a Brownian Motion, the measure $P^\Phi$ which $W^\Phi$ induces on $\overline{\Omega}$ agrees with $\overline{P}$.  Thus, $\overline{P}(\{\overline{\omega} : \overline{W}_{\cdot \wedge \overline{\tau}(\overline{\omega})} \in A \}) = P^\Phi \left( W^\Phi \left( \{ \omega : W^\Phi_{\cdot \wedge \tau(\omega)}(\omega) \in A \} \right) \right) = P^{\mb{N}} \left(\{ \omega : W^\Phi_{\cdot \wedge \tau(\omega)}(\omega) \in A \} \right)$.  Thus, $\overline{W}_{\cdot \wedge \overline{\tau}}$ and $W^\Phi_{\cdot \wedge \tau}$ are identically distributed.

Conversely, suppose that $\overline{\tau}$ is an $\overline{\mb{F}}$-stopping time.  Define $\tau:\Omega^{\mb{N}} \ra \mb{R}$ by $\tau = \overline{\tau} \circ W^\Phi$.  We claim that $\tau$ is a stopping time.  Let $t \in \mb{R}$.  Then
\[
\begin{split}
\{ \omega : \tau(\omega) \leq t \}
&= \{\omega : \overline{\tau}(W^\Phi(\omega)) \leq t \}
\\&= (W^\Phi)^{-1} \left( \{ \overline{\omega} : \overline{\tau}(\overline{\omega}) \leq t \} \right).
\end{split}
\]

Since $\overline{\tau}$ is an $\overline{\mb{F}}$-stopping time, the set $\{ \overline{\omega} : \overline{\tau}(\overline{\omega}) \leq t \} \in \overline{\mc{F}}_t$, and $(W^\Phi)^{-1} \left( \{ \overline{\omega} : \overline{\tau}(\overline{\omega}) \leq t \} \right) \in \mc{F}^\Phi_t$.  So, $\tau \in \mc{T}^\Phi$.  

Now we claim that $W^\Phi_{\cdot \wedge \tau}$ and $\overline{W}_{\cdot \wedge \overline{\tau}}$ are identically distributed as processes.  Let $A \in \overline{\mc{F}}^\infty$.  
\[
\begin{split}
\{ \omega : W^\Phi_{\cdot \wedge \tau(\omega)} \in A\}
& = \{ \omega : \overline{W}_{\cdot \wedge \overline{\tau}(\overline{\omega})}(\overline{\omega}) \in A \text{ for } \overline{\omega} \text{ such that } W^\Phi(\omega) = \overline{\omega} \}
\\& = (W^\Phi)^{-1} \left( \{ \overline{\omega} : \overline{W}_{\cdot \wedge \overline{\tau}(\overline{\omega})}(\overline{\omega}) \in A \} \right).
\end{split}
\]

As before, 
\[
\begin{split}P^{\mb{N}} \left( \{ \omega : W^\Phi_{\cdot \wedge \tau(\omega)} \in A\} \right) 
&= P^{\mb{N}} \left( (W^\Phi)^{-1} \left( \{ \overline{\omega} : \overline{W}_{\cdot \wedge \overline{\tau}(\overline{\omega})}(\overline{\omega}) \in A \} \right) \right) 
\\&= P^\Phi \left( \{ \overline{\omega} : \overline{W}_{\cdot \wedge \overline{\tau}(\overline{\omega})}(\overline{\omega}) \in A \} \right) 
\\&= \overline{P} \left( \{ \overline{\omega} : \overline{W}_{\cdot \wedge \overline{\tau}(\overline{\omega})}(\overline{\omega}) \in A \} \right),
\end{split}
\]
and so the processes are identically distributed.
\end{proof}

\begin{lemma}\label{lemma4} For each $\Phi \in \mathfrak{S}$, there exists a sequence of $\overline{\mb{F}}$-stopping times $\overline{\Phi} = \{\overline{\phi_1} = 0,\overline{\phi_2},\ldots \}$ such that for 
\begin{equation}\label{piline}
\overline{\pi}^{\overline{\Phi}}_t \triangleq \hat{\pi} + \int_0^t \overline{\pi}^{\overline{\Phi}}_s (1-\overline{\pi}^{\overline{\Phi}}_s) d\overline{W}_s + \sum_{i=1}^\infty (\hat{\pi} - \overline{\pi}^{\overline{\Phi}}_{\overline{\phi}_i-}) 1_{\{t \geq \overline{\phi_i}\}},
\end{equation}
$\overline{\pi}^{\overline{\Phi}}$ is identically distributed with $\pi^\Phi$.
\end{lemma}

\begin{proof}

Let $N^\Phi$ be the simple point process on $\Omega^{\mb{N}}$ which jumps at the $\mb{F}^\Phi$-stopping times $\phi_1,\phi_2,\ldots$.  Let $\overline{\Phi} = \{\overline{\phi}_1,\overline{\phi}_2,\ldots\}$ be a sequence of $\overline{\mb{F}}$-stopping times whose existence is guaranteed by Lemma \ref{lemma3}, and let $\overline{N}$ be the simple point process on $\overline{\Omega}$ which jumps at the $\overline{\mb{F}}$-stopping times $\overline{\phi}_1,\overline{\phi}_2,\ldots$.  According to Lemma \ref{lemma3}, $(W^\Phi,N^\Phi)$ and $(\overline{W},\overline{N})$ are identically distributed as processes.  Let $f(x) = x(1-x)$ and let $g(x) = \hat{\pi} - x$.  Then $\pi^\Phi$ and $\overline{\pi}^{\overline{\Phi}}$ satisfy the SDE's
\[
d \pi^\Phi_t = f(\pi^\Phi_t)d W^\Phi_t + g(\pi^\Phi_{t-}) d N^\Phi_t,
\]
and
\[
d \overline{\pi}^{\overline{\Phi}}_t = f(\overline{\pi}^{\overline{\Phi}}_t)d \overline{W}_t + g(\overline{\pi}^{\overline{\Phi}}_{t-})d \overline{N}_t.
\]

Note that if $\pi^\Phi$ starts inside the interval $(0,1)$, then it stays there for all time, and similarly for $\overline{\pi}^{\overline{\Phi}}$.  On the interval $(0,1)$, $f(x)$ is bounded and Lipschitz, and the same goes for $g(x)$.  By Theorem $9.1$ of \cite{MR1011252} (see p. $245-6$), the above SDE's have uniqueness in law.  Consequently, $\pi^\Phi$ and $\overline{\pi}^{\overline{\Phi}}$ are identically distributed.  
\end{proof}

The converse is proven similarly using Lemma \ref{lemma3}.

\begin{lemma}\label{lemma5}  Let $\overline{\Phi} = \{\overline{\phi}_1,\overline{\phi}_2,\ldots\}$ be a collection of $\overline{\mb{F}}$-stopping times which increase to infinity, and let $\overline{\Phi}$ induce $\overline{\pi}^{\overline{\Phi}}$ as in (\ref{piline}).  Then there exists $\Phi \in \mathfrak{S}$ such that $\pi^\Phi$ is identically distributed to $\overline{\pi}^{\overline{\Phi}}$.
\end{lemma}

\begin{lemma}\label{lemma6} We have
\begin{equation}\label{lemma6eq}
V_{\hat{\pi}} = \underset{\overline{\Phi} \in \overline{\mathfrak{S}}}{\inf} \ \underset{\tau \in \overline{\mc{T}}}{\inf} E[c \tau + (1-\overline{\pi}^{\overline{\Phi}}_\tau)],
\end{equation}
where $\overline{\mathfrak{S}}$ is the set of all sequences $\overline{\Phi} = \{\overline{\phi}_1,\overline{\phi}_2,\ldots\}$ of $\overline{\mb{F}}$-stopping times which increase to infinity.
\end{lemma}

\begin{proof}  Denote by $\overline{V}_{\hat{\pi}}$ the right side of \eqref{lemma6eq}.  Let $\Phi \in \mathfrak{S}$, and consider the optimal stopping problem $\underset{\tau \in \mc{T}^\Phi}{\inf} \ E[c \tau + (1-\pi^\Phi_\tau)]$.  By Lemma \ref{lemma4}, the process $\pi^\Phi_t$ is distributed identically to $\overline{\pi}^{\overline{\Phi}}$.  According to Lemma $2.3$ of \cite{MR1809524}, the value function associated to the optimal stopping of a process depends only on that process's distribution.  Therefore $\underset{\tau \in \mc{T}^\Phi}{\inf} \ E[c \tau + (1-\pi^\Phi_\tau)] = \underset{\tau \in \overline{\mc{T}}}{\inf} \ E[c \tau + (1-\overline{\pi}^{\overline{\Phi}}_\tau)] \geq \overline{V}_{\hat{\pi}}$.  Taking the infimum over all $\Phi \in \mathfrak{S}$, we obtain
\[
V_{\hat{\pi}} \geq \overline{V}_{\hat{\pi}}.
\]

Now, let $\overline{\Phi} \in \overline{\mathfrak{S}}$.  By Lemma \ref{lemma5}, there exists $\Phi \in \mathfrak{S}$ such that $\pi^\Phi$ is identically distributed with $\overline{\pi}^{\overline{\Phi}}$.  So, using the same reasoning as above and taking the infimum over all $\overline{\Phi}$, we obtain
\[
\overline{V}_{\hat{\pi}} \geq V_{\hat{\pi}}.
\]

\end{proof}

\section{Working with the new problem, and reduction to an optimal stopping problem}

From now on, we will drop the overline notation, and simply write $\mathfrak{S}, \Phi, \pi_t, W, \mb{F}, \mc{T},$ for, respectively, the set of allowed switching strategies, an arbitrary switching strategy, the posterior process, the single Brownian Motion $W$, the filtration induced by $W$, and the stopping times for that filtration.

Let $\pi^0$ denote the posterior process when there is no switching.  In other words, $\pi^0$ satisfies the SDE $d\pi^0_t = \pi^0_t(1-\pi^0_t)dW_t$ along with $\pi^0={\hat{\pi}}$.  We next define the reflected process $\pi^r$ with boundary at $\hat{\pi}$:

\begin{equation}\label{eq2.1}
d\pi^r_t = \pi^r_t(1-\pi^r_t)dW_t + dA_t, 
\end{equation}

\noindent where $A_t$ is continuous, non-decreasing, flat off of $\pi^r = \hat{\pi}$, $A_0 = 0$.

We also have an optimal stopping problem associated with $\pi^r$:
\begin{equation}
V_{\hat{\pi}}^r = \underset{\tau \in \mc{T}}{\inf} \ E[c \tau + (1-\pi^r_\tau)]
\end{equation}

\begin{lemma}\label{lemma2.1}  $V^r_{\hat{\pi}} \leq V_{\hat{\pi}}$.
\end{lemma}

\begin{proof} Let $\Phi = \{\phi_1,\phi_2,\ldots\} \in \mathfrak{S}$.  Fix $i$; we will show that $\pi^r_t \geq \pi^\Phi_t$ on $[\phi_i,\phi_{i+1})$.  By construction, $\pi^\Phi_{\phi_i} = \hat{\pi}$, and on the interval $[\phi_i,\phi_{i+1})$, the dynamics of $\pi^\Phi$ are described by the diffusion $d \pi^\Phi_t = \pi^\Phi_t (1-\pi^\Phi_t) dW_t$.  Let $\pi^{0,\phi_i}$ be the un-switched diffusion starting from $\pi^0_{\phi_i} = \hat{\pi}$, so that $\pi^{0,\phi_i} = \pi^\Phi$ on $[\phi_i,\phi_{i+1})$.  Note that, by construction, $\pi^r_{\phi_i} \geq \hat{\pi}$.  Then \eqref{eq2.1} and the comparison theorem for SDE's (i.e. Theorem $54$ p. 324 of \cite{MR2273672}) imply that $\pi^r \geq \pi^{0,\phi_i}$ on $[\phi_i,\phi_{i+1})$, and so $\pi^r \geq \pi^\Phi$ on $[\phi_i,\phi_{i+1})$.  It now follows that $\pi_t^r \geq \pi_t^\Phi$ for all $t$, a.s.  Consequently, for any $\tau \in \mc{T}$, $E[(1-\pi^r_\tau)] \leq E[(1-\pi^\Phi_\tau)]$, implying that $V^r_{\hat{\pi}} \leq V_{\hat{\pi}}$.
\end{proof}


Following \cite{MR1780932}, p. 146, we give the Skorokhod representation of $\pi^r$.  Given a process $Y$ and $\hat{\pi} \in \mb{R}$, the Skorokhod representation consists in finding a process $X$ and an increasing process $A$ such that $X = Y + A$, $X \geq \hat{\pi}$, and $\int_0^\infty (X_s-\hat{\pi})dA_s = 0$, i.e. $A$ only increases when $X = \hat{\pi}$.

Let $\sigma(x) = (1-x)x$, and let $Y$ solve the SDE 
\[
Y_t \triangleq \hat{\pi} + \int_0^t \sigma(Y_s + A_s(Y))dW_s,
\]
\[
A_t(Y) \triangleq \underset{0 \leq s \leq t}{\sup} \left\{ (Y_s - \hat{\pi})^- \right\}.\]  
As in \cite{MR1780932}, the SDE does in fact have a unique strong solution.  Then, if we set $X_t \triangleq Y_t + A_t(Y)$, it is clear that $\pi^r = X$.

Let $\epsilon>0$.  We outline a parametrized family of switching strategies (impulse controls).  Let $\Phi^\epsilon$ denote the strategy that switches channels whenever the observed posterior process hits the level $\hat{\pi} - \epsilon$.  $\Phi^\epsilon$ induces the process $\pi^\epsilon$, starting from $\pi^\epsilon_0 = \hat{\pi}$, which diffuses according to $d \pi^\epsilon_t = \pi^\epsilon_t (1-\pi^\epsilon_t) dW_t$ on $(\hat{\pi} - \epsilon,1)$.  When it reaches the level $\hat{\pi} - \epsilon$, it is instantaneously brought back to $\hat{\pi}$ (i.e. switched).  We wish to give a Skorokhod type representation of $\pi^\epsilon$.  Consider the SDE 
\begin{equation}\label{Yepsilon}
Y^\epsilon_t \triangleq \hat{\pi} + \int_0^t \sigma(Y^\epsilon_s + A^\epsilon_s(Y^\epsilon))dW_s,
\end{equation}
where 
\[
A^\epsilon_s(Y^\epsilon) \triangleq \epsilon \left \lfloor \frac{1}{\epsilon} \underset{0 \leq s \leq t}{\sup} \left \{ (Y^\epsilon_s - \hat{\pi})^- \right \} \right \rfloor.
\]  
Note that $A_s^\epsilon(\cdot)$ is not even continuous with respect to the uniform norm on continuous paths.  Therefore, the standard theory does not imply that the SDE \eqref{Yepsilon} has a strong solution.  We can, however, show that a solution exists by a piecewise construction.

\begin{lemma}\label{lemma2.2} For each $\epsilon>0$, the SDE \eqref{Yepsilon} has a strong solution.  Moreover, for $X^\epsilon_t = Y^\epsilon_t + A^\epsilon_t(Y^\epsilon)$, $X^\epsilon = \pi^\epsilon$.
\end{lemma}

\begin{proof}  Consider the SDE
\begin{equation}\label{eq2.2.1}
Y^1_t = Y^{\epsilon,1}_t \triangleq \hat{\pi} + \int_0^t \sigma(Y^{\epsilon,1}_s)dW_s.
\end{equation}
As $\sigma(\cdot)$ is Lipschitz and bounded on the interval $(0,1)$, it is known (see Theorem $11.5$ of \cite{MR1780932}) that \eqref{eq2.2.1} has a strong solution.  Let $\tau^{\epsilon,0} \triangleq 0$, and $\tau^{\epsilon,1} \triangleq \inf \{t \geq 0 : Y^{\epsilon,1}_t = \hat{\pi} - \epsilon \}$.  Note that on the random time interval $[0,\tau^{\epsilon,1})$, $Y^{\epsilon,1}$ solves the SDE \eqref{Yepsilon}.  For $t \geq \tau^{\epsilon,1}$, consider next the SDE
\begin{equation}\label{eq2.2.2}
Y^{\epsilon,2}_t \triangleq Y^{\epsilon,1}_{\tau^{\epsilon,1}} + \int_{\tau^{\epsilon,1}}^t \sigma(Y^{\epsilon,2}_s + \epsilon)dW_s.
\end{equation}
As before, \eqref{eq2.2.2} has a strong solution.  Let $\tau^{\epsilon,2} \triangleq \inf \{ t \geq \tau^{\epsilon,1} : Y^{\epsilon,2}_t = \hat{\pi} - 2\epsilon \}.$  Then on $[\tau^{\epsilon,1},\tau^{\epsilon,2})$, $Y^{\epsilon,2}$ solves \eqref{Yepsilon}.  Arguing inductively, we define $Y^{\epsilon,n}_t$, for $t \geq \tau^{\epsilon,n-1}$, by
\begin{equation}\label{eq2.2.3}
Y^{\epsilon,n}_t \triangleq Y^{\epsilon,n-1}_{\tau^{\epsilon,n-1}} + \int_{\tau^{\epsilon,n-1}}^t \sigma(Y^{\epsilon,n}_s + (n-1)\epsilon)dW_s,
\end{equation}
which has a strong solution as before, and the stopping time $\tau^{\epsilon,n} \triangleq \inf \{ t \geq \tau^{\epsilon,n-1} : Y^{\epsilon,n} = \hat{\pi} - n\epsilon\}$.  Defining the process $Y^\epsilon$ by $Y^\epsilon_t \triangleq Y^{\epsilon,n}_t$ for $t \in [\tau^{\epsilon,n-1},\tau^{\epsilon,n})$, $n \geq 1$, it is apparent that $Y^\epsilon$ solves the SDE \eqref{Yepsilon}.  

For the last claim, note that when $A^\epsilon(Y^\epsilon)$ is constant, $dX^\epsilon_t = \sigma(Y^\epsilon_t + A^\epsilon_t(Y^\epsilon))dW_t = \sigma(X^\epsilon_t)dW_t$.  The times when $A^\epsilon(Y^\epsilon)$ jumps (by $\epsilon$) correspond to the impulses from $\hat{\pi} - \epsilon$ to $\hat{\pi}$.
\end{proof}

\begin{lemma}\label{lemma2.3}  For any $t,\epsilon >0$, we have $E \left[(Y - Y^\epsilon)^{*2}_t\right] \leq 8t e^{32t} \epsilon^2$.  In particular, for any $t \geq 0$, $(Y - Y^\epsilon)^*_t \ra 0$ in $L^2$ as $\epsilon \ra 0$.
\end{lemma}

\begin{proof} Let $K$ be the Lipschitz constant of $\sigma(\cdot)$ on $(0,1)$.  Write
\[
\begin{split}
E (Y^\epsilon - Y)^{*2}_t 
& = E \left[ \left( \hat{\pi} + \int_0^t \sigma(Y_s^\epsilon + A_s^\epsilon(Y^\epsilon)) dW_s - \hat{\pi} - \int_0^t \sigma(Y_s + A_s(Y))dW_s \right)^{*2} \right]
\\& = E \left[\Big( \sigma(Y^\epsilon_\cdot + A^\epsilon_\cdot(Y^\epsilon)) \cdot W - \sigma(Y^\epsilon_\cdot + A_\cdot(Y^\epsilon))\cdot W + \sigma(Y^\epsilon_\cdot + A_\cdot(Y^\epsilon)) \cdot W - \sigma(Y_\cdot + A_\cdot(Y)) \cdot W \Big)^{*2}_t \right]
\\& \leq E \Bigg[ \Big(\big(\sigma(Y^\epsilon_\cdot + A^\epsilon_\cdot(Y^\epsilon)) \cdot W - \sigma(Y^\epsilon_\cdot + A_\cdot(Y^\epsilon)) \cdot W \big)^*_t 
\\& \ \ \ \ \ \ + \big(\sigma(Y^\epsilon_\cdot + A_\cdot(Y^\epsilon)) \cdot W - \sigma(Y_\cdot + A_\cdot(Y)) \cdot W \big)^*_t \Big)^2 \Bigg]
\\& \leq 2E \left[ \big( \left( \sigma(Y^\epsilon_\cdot + A^\epsilon_\cdot(Y^\epsilon)) - \sigma(Y^\epsilon_\cdot + A_\cdot(Y^\epsilon)) \right) \cdot W  \big)^{*2}_t \right] 
\\& \ \ \ \ \ \ + 2E \left[\big( \left( \sigma(Y^\epsilon_\cdot + A_\cdot(Y^\epsilon)) - \sigma(Y_\cdot + A_\cdot(Y)) \right) \cdot W \big)^{*2}_t \right]
\\& \triangleq (1) + (2),
\end{split} 
\]
with the second inequality above following from $(a+b)^2 \leq 2a^2 + 2b^2$.  Next, using the Burkholder-Davis-Gundy Theorem for the first inequality and the $K$-Lipschitzness of $\sigma(\cdot)$ for the second, 
\[
\begin{split}
(1)
& \leq 2C_2 E \int_0^t \left( \sigma(Y^\epsilon_s + A^\epsilon_s(Y^\epsilon)) - \sigma(Y^\epsilon_s + A_s(Y^\epsilon)) \right)^2 ds
\\& \leq 2C_2 K^2 E \int_0^t (Y^\epsilon_s + A^\epsilon_s(Y^\epsilon) - Y^\epsilon_s - A_s(Y^\epsilon))^2 ds
\\& \leq 2C_2 K^2 \epsilon^2 t;
\end{split}
\]
\noindent the last inequality follows from the fact that for a given path $\omega$, $\underset{0 \leq s \leq t}{\sup} \ |A_s(\omega) - A_s^\epsilon(\omega)| \leq \epsilon$.  The constant $C_2$ is a universal constant arising from the Burkholder-Davis-Gundy Theorem.  The $L^2$ version used above actually can be proven using Doob's $L^2$-inequality for martingales, and from this the explicit formula $C_2 = 4$ can be derived.  For details, see p. $14$ of \cite{MR1121940}.

 Next, we note that $A_\cdot(\cdot)$ is Lipschitz continuous with respect to the uniform norm on continuous paths, with Lipschitz constant $1$.  Applying this fact for the third inequality below, Burkholder-Davis-Gundy for the first inequality, the $K$-Lipschitz continuity of $\sigma(\cdot)$ for the second inequality, and Fubini's Theorem in the last inequality, we obtain
\[
\begin{split}
(2)
& \leq 2C_2 E \int_0^t \left( \sigma(Y^\epsilon_s + A_s(Y^\epsilon)) - \sigma(Y_s + A_s(Y)) \right)^2 ds
\\& \leq 2C_2 K^2 E \int_0^t (Y^\epsilon_s + A_s(Y^\epsilon) - Y_s + A_s(Y))^2 ds
\\& \leq 8C_2 K^2 E \int_0^t (Y^\epsilon_s - Y_s)^2 ds
\\& \leq 8 C_2 K^2\int_0^t E (Y^\epsilon - Y)^{*2}_s ds.
\end{split}
\]

For each $\epsilon>0$, define $f^\epsilon:\mb{R}_+ \ra \mb{R}_+$ by $f^\epsilon(s) = E (Y^\epsilon - Y)^{*2}_s$.  According to the above reasoning, $f^\epsilon(t) \leq 2C_2 K^2 t \epsilon^2 + 8 C_2 K^2 \int_0^t f^\epsilon(s) ds$.  By Gronwall's Lemma, it follows then that $f^\epsilon(t) \leq 2C_2 K^2 t \epsilon^2 e^{8 C_2 K^2 t}$.  Since all processes in question live in the interval $(0,1)$, we may assume that $K=1$.  Therefore, $f^\epsilon(t) \leq 8te^{32t} \epsilon^2$.  In particular, $t$, $(Y^\epsilon - Y)^*_t \ra 0$ in $L^2$ as $\epsilon \ra 0$.  
\end{proof}

\begin{corollary}\label{cor2.4}  For any $t,\epsilon \geq 0$ $E \left[(\pi^\epsilon - \pi^r)^{*2}_t \right] \leq 16te^{32t} \epsilon^2 + \epsilon$.  In particular, as $\epsilon \ra 0$, $(\pi^\epsilon - \pi^r)^*_t \ra 0$ in $L^2$.  
\end{corollary}

\begin{proof}  Write $\pi^r_t = Y_t + A_t(Y)$ and $\pi^\epsilon_t = Y^\epsilon_t + A^\epsilon_t(Y^\epsilon)$.  We have shown in Lemma \ref{lemma2.3} that $(Y^\epsilon - Y)^*_t \ra 0$ in $L^2$ as $\epsilon \ra 0$.  Therefore, it suffices to show that $(A_\cdot(Y) - A^\epsilon_\cdot(Y^\epsilon))^*_t \ra 0$ in $L^2$ as $\epsilon \ra 0$.  So, for any $s \geq 0$,
\[
\begin{split}
|A_s(Y) - A^\epsilon_s(Y^\epsilon)|
& \leq |A_s(Y) - A_s(Y^\epsilon)| + |A_s(Y^\epsilon) - A_s^\epsilon(Y^\epsilon)|
\\& \leq (Y - Y^\epsilon)^*_s + \epsilon,
\end{split}
\]

\noindent where we have used the Lipschitz continuity of $A_s(\cdot)$ with respect to the uniform norm.  Therefore, $(A_\cdot(Y) - A^\epsilon_\cdot(Y^\epsilon))^*_t \leq (Y - Y^\epsilon)^*_t + \epsilon$, which converges to 0 in $L^2$ as $\epsilon \ra 0$.  The quantitative estimate is also clear.
\end{proof}

\begin{lemma}\label{lemma2.2} $V^r_{\hat{\pi}} = V_{\hat{\pi}}$
\end{lemma}

\begin{proof} In light of Lemma \ref{lemma2.1}, it suffices to show that $V^r_{\hat{\pi}} \geq V_{\hat{\pi}}$.  Without loss of generality, we assume that $V^r_{\hat{\pi}} < \infty$; otherwise, there is nothing to show.  Let $\{\tau_n : n \in \mb{N}\}$ be a sequence of stopping times such that $E[c \tau_n + (1-\pi^r_{\tau_n})] \downarrow V^r_{\hat{\pi}}$.  Fix $\delta > 0$, and choose $n$ sufficiently large so that $E \left[c \tau_n + (1-\pi^r_{\tau_n}) \right] < V^r_{\hat{\pi}} + \delta$.  

Next, we note that the processes $\pi^r$ and $\pi^\epsilon$ are all bounded, so that in particular, they are uniformly of Class D.  Therefore, for a suitably large $t$, it is the case that $E \left[\pi^r_{\tau_n} 1_{\{\tau_n > t\}} \right], E \left[\pi^\epsilon_{\tau_n} 1_{\{\tau_n > t\}} \right] < \delta$ for each $\epsilon > 0$.  By Corollary \ref{cor2.4}, for $\epsilon$ sufficiently small, $ \left| E \left[\pi^r_{\tau_n} 1_{\{\tau_n \leq t\}} \right] - E \left[\pi^\epsilon_{\tau_n} 1_{\{\tau_n \leq t\}} \right] \right| < \delta$.  Thus, 
\[
\begin{split}
\left| V^r_{\hat{\pi}} - E \left[c \tau_n + (1 - \pi^\epsilon_{\tau_n}) \right] \right| 
& \leq \left| V^r_{\hat{\pi}} - E \left[ c \tau_n + (1-\pi^r_{\tau_n}) \right] \right| + \left| E \left[c \tau_n + (1 -\pi^r_{\tau_n}) \right] - E \left[c \tau_n + (1- \pi^\epsilon_{\tau_n}) \right] \right|
\\& = \left| V^r_{\hat{\pi}} - E \left[ c \tau_n + (1-\pi^r_{\tau_n}) \right] \right| + \left| E \left[\pi^r_{\tau_n} \right] - E \left[\pi^\epsilon_{\tau_n}  
\right] \right|
\\& < \delta + \left| E \left[\pi^r_{\tau_n} 1_{\{\tau_n > t\}} \right] - E \left[\pi^\epsilon_{\tau_n} 1_{\{\tau_n > t\}} \right] \right| + \left| E \left[\pi^r_{\tau_n} 1_{\{\tau_n \leq t\}} \right] - E \left[\pi^\epsilon_{\tau_n} 1_{\{\tau_n \leq t\}} \right] \right|
\\& < \delta + 2\delta + \delta
\\&= 4\delta.
\end{split}
\]

Since $V_{\hat{\pi}} \leq V^{\Phi^\epsilon}_{\hat{\pi}} \leq E \left[c \tau_n + (1-\pi^\epsilon_{\tau_n}) \right]$, it now follows that $V_{\hat{\pi}} \leq V^r_{\hat{\pi}}$.
\end{proof}

\section{Optimal stopping of the reflected diffusion}

We wish to relate the optimal stopping problem $V^r_{\hat{\pi}} = \underset{\tau \in \mc{T}}{\inf} \ E[c\tau + (1-\pi^r_\tau)]$ to an ODE with a free boundary.  First, we look for $f:[\hat{\pi},1] \ra \mb{R}$ and $\pi^* \in (\hat{\pi},1)$ that satisfy:

\begin{equation}\label{eq3.1}
\frac{1}{2}[x(1-x)]^2 \frac{d^2f}{dx^2} = -c, \hat{\pi}<x<\pi^*,
\end{equation}
\begin{equation}\label{eq3.2}
f(x) = 1 - x, \pi^* \leq x \leq 1,
\end{equation}
\begin{equation}\label{eq3.3}
f'(\hat{\pi}) = 0, f'(\pi^*) = -1.
\end{equation}

Notice in particular that we require $f$ to be $C^1$ at $\hat{\pi}$ and $\pi^*$; $\pi^*$ must be chosen to ensure that this happens.

\begin{lemma}\label{lemma2.3} The problem \eqref{eq3.1},\eqref{eq3.2},\eqref{eq3.3} has a unique solution, for precisely one $\pi^* \in [\hat{\pi},1)$.
\end{lemma}

\begin{proof} One may verify directly that the function $\Psi_{A,B}(x) \triangleq 2c(1-2x)\log\frac{x}{1-x} + Ax + B \triangleq \Psi(x) + Ax + B$, for constants $A$ and $B$, is the general solution of \eqref{eq3.1}.  We will show that the boundary conditions are satisfied for precisely one $\pi^*$, $\overline{A}$, and $\overline{B}$.

The condition $f'(\hat{\pi}) = 0$ forces $\overline{A} = -\Psi'(\hat{\pi}) = -2c \left[ \frac{2 \hat{\pi} - 2(\hat{\pi} -1)\hat{\pi}\log \left(\frac{\hat{\pi}}{1-\hat{\pi}}\right) -1}{(\hat{\pi}-1)\hat{\pi}}\right]$.  Since $\Psi(x) + \overline{A}x$ is strictly concave and $\Psi'(x) + \overline{A}$ is continuous on $[\hat{\pi},1)$, $\Psi'(\hat{\pi}) + \overline{A} = 0$, and $\underset{x \uparrow 1}{\lim} \ \Psi(x) + \overline{A}x = - \infty$ (so $\underset{x \uparrow 1}{\lim} \ \Psi'(x) + \overline{A} = -\infty$), it follows that there is a unique $\pi^* \in (\hat{\pi},1)$ such that $\Psi'(\pi^*) + \overline{A} = -1$.  Define $\overline{B}$ so that $\overline{B}$ satisfies the equality $1 - \pi^* = \Psi_{\overline{A},\overline{B}}(\pi^*) = \Psi(\pi^*) + \overline{A} \pi^* + \overline{B}$.  Taking $f(x) = \Psi_{\overline{A},\overline{B}}(x)$ for $x \in [\hat{\pi},\pi^*)$ and $f(x) = 1-x$ for $x \in [\pi^*,1]$ yields the unique solution to \eqref{eq3.1},\eqref{eq3.2},\eqref{eq3.3}.
\end{proof}

Our candidate for the value function is therefore
\begin{equation}\label{valfun}
f(x) \triangleq 
\begin{cases}
2c(1-2x)\log \frac{x}{1-x} + \overline{A}x + \overline{B} & \text{if } 0 \leq x \leq \pi^*
\\
1 - x & \text{if } \pi^* \leq x \leq 1
\end{cases}
\end{equation}
with
\[
 \overline{A} = -2c \left[ \frac{2 \hat{\pi} - 2(\hat{\pi} -1)\hat{\pi}\log \left(\frac{\hat{\pi}}{1-\hat{\pi}}\right) -1}{(\hat{\pi}-1)\hat{\pi}}\right],
\]

\[
\pi^* \mbox{ satisfying } 2c \left[ \frac{2 \pi^* - 2(\pi^* -1)\pi^*\log \left(\frac{\pi^*}{1-\pi^*}\right) -1}{(\pi^*-1)\pi^*}\right] + \overline{A} = -1,
\]
\[
\overline{B} \mbox{ satisfying } 1 - \pi^* = 2c(1-2\pi^*) + \overline{A}\pi^* + \overline{B}.
\]

For each $x \in [\hat{\pi},1]$, we set $V^r_{\hat{\pi}}(x) = \underset{\tau \in \mc{T}}{\inf} \ E_x \left[c \tau + (1-\pi^r_\tau)\right]$, where the expectation $E_x[\cdot]$ denotes expectation under the probability $P_x$, i.e. $P_x(\pi^r_0 = x) = 1$.  We now claim that $f(x)$ is equal to the value function $V^r_{\hat{\pi}}(x)$.  Consider the set $\mc{D} \triangleq \{ f \in C^2_b([\hat{\pi},1)) : f'(\hat{\pi}) = 0 \}$.  The infinitesimal generator $\mc{L}^r$ of $\pi^r$ satisfies, for $f \in \mc{D}$, $\mc{L}^r f(x) = \frac{1}{2}x^2(1-x)f''(x)$.

\begin{lemma}\label{lemma2.4} For $f(x)$ as above, $V^r_{\hat{\pi}}(x) = f(x)$. 
\end{lemma}

\begin{proof}  

We wish to apply a verification theorem for optimal stopping problems, Theorem $10.4.1$ of \cite{MR2001996}, p. $225$.  We must check that $f(x)$ defined by \eqref{eq3.1},\eqref{eq3.2},\eqref{eq3.3} satisfies the nine hypotheses of that theorem.  Note that several inequalities are reversed because our problem involves a minimization over all stopping times.  Let $G = [\hat{\pi},1]$, and let $D = \{x \in G : f(x) < 1 -x\}$.

\begin{itemize}
\item [(i)] $f \in C^1(G)$: This is true by construction.
\item [(ii)] $f \leq 1 -x$ on $G$: At $\pi^*$, $f(\pi^*) = 1-\pi^*$.  Since $f(x)$ is concave down, $f(x) \leq 1 - x$ for $x \in [\hat{\pi},\pi^*]$, and by construction $f(x) = 1-x$ on $[\pi^*,1]$.
\item [(iii)] $E_x \left[ \int_0^\infty 1_{\{\pi^*\}}(\pi^r_s)ds \right] = 0$: This follows from the fact that the speed measure of $\pi^r_s$ is $m^r(dx) \triangleq \frac{dx}{x^2(1-x)^2}$.  Now, i.e. Proposition $3.10$ of \cite{MR1303781}, p. $307$ may be applied.
\item [(iv)] $\partial D$ is Lipschitz: This is trivial in the one-dimensional problem here.
\item [(v)] $f \in C^2(G \setminus \{\pi^*\})$ and the second order derivatives of $f$ are bounded near $\pi^*$: For $x \in (\hat{\pi},\pi^*)$, $f''(x) = \frac{-2c}{x^2(1-x)^2}$, which is bounded on $(\hat{\pi},\pi^*)$, and for $x \in (\pi^*,1)$, $f''(x) = 0$.
\item [(vi)] $\mc{L}^r f + c \geq 0$ on $G \setminus D$:  For $x \in G \setminus D$, $\mc{L}^rf + c = 0 + c \geq 0$.
\item [(vii)] $\mc{L}^r f + c = 0$ on $D$:  For $x \in D$, $\mc{L}^r f + c = \frac{1}{2}x^2(1-x)^2\left( \frac{-2c}{x^2(1-x)^2} \right) + c = 0$.
\item [(viii)] $\tau_D \triangleq \inf \{ t > 0 : \pi^r_t \not \in D\} < \infty$, $P_x$-a.s. for each $x \in G$:  Using the same argument as in $(iii)$, Proposition $3.10$ of \cite{MR1303781} implies that $E_x[\tau_D] < \infty$ for each $x \in G$.
\item [(ix)] The family $\{\pi^r_\tau : \tau \leq \tau_D, \tau \in \mc{T}\}$ is $P_x$-uniformly integrable for any $x \in G$: This is immediate, using the fact that $\pi^r$ is bounded.
\end{itemize}

Having checked all the hypotheses of the verification theorem, we deduce that $f(x) = V^r_{\hat{\pi}}(x)$.
\end{proof}

\section{A rough algorithm for quickest search}

Using the methods of the previous sections, we can describe near optimal algorithms for quickly finding a channel which satisfies hypothesis $H_1$.  We outline a procedure below for finding an $\bm{\epsilon}$-optimal strategy.
\begin{itemize}
\item[(1)] Fix $\bm{\epsilon} >0$.
\item[(2)] For given values of $c,\hat{\pi}$, calculate the threshold $\pi^* = \pi^*(c,\hat{\pi})$ via Lemma \ref{lemma2.3}.  Let $\tau \triangleq \inf \{t \geq 0 : \pi^r_t = \pi^*\}$ be defined for any version of $\pi^r$.
\item[(3)] Choose $t>0$ sufficiently large so that $P(\tau > t) < \frac{\bm{\epsilon}}{4}$.  This can be done, for example, by calculating $E[\tau]$ via the speed measure of $\pi^r$.
\item[(4)] Choose $\epsilon_2$ sufficiently small so that $16te^{32t}\epsilon_2^2 + \epsilon_2< \frac{\bm{\epsilon}}{2}$.
\item[(5)] Adopt the switching strategy $\Phi^{\epsilon_2}$, in which the observed channel is switched whenever the posterior level hits $\hat{\pi} - \epsilon_2$.  
\item [(6)] The switching strategy $\Phi^{\epsilon_2}$ induces the observed Brownian Motion $W = W^{\Phi^{\epsilon_2}}$.  Using $W$, construct the solution to the SDE $Y_t = \hat{\pi} + \int_0^t \sigma(Y_s + A_s(Y))dW_s$, and set $X_t = Y_t + A_t(Y)$.  Let $\tau^* = \inf \{t \geq 0 : X_t = \pi^*\}$.
\item[(7)] At time $\tau^*$, accept hypothesis $H_1$ for the channel which is currently being observed.  
\end{itemize}
Applying the reasoning of Lemmas \ref{lemma2.3} and \ref{lemma2.2}, we may deduce that this observation/stopping strategy will be $\bm{\epsilon}$-optimal.

\section{Numerical Results}

In this section, we illustrate our previous results by computing the optimal threshold level for various levels of the observation cost $c$ and prior $\hat{\pi}$.  The data below, which can be found in the second appendix, is directly calculated from the value function established in Section $4$.  We first plot the threshold levels against the observation cost $c$, when the prior $\hat{\pi}$ is fixed.  As indicated by Tables \ref{fig:table1} and \ref{table2} below, for fixed $\hat{\pi}$, $\pi^*(c)$ decreases with $c$.  This is not surprising, because the higher the running cost for observations, the lower one's standards will be for accepting the hypothesis that a channel satisfies $H_1$.  

Next, we plot the threshold levels against the prior $\hat{\pi}$, when the observation cost $c$ is fixed.  As indicated by Tables ~\ref{table3} and ~\ref{table4} below, for fixed $c$, $\pi^*(\hat{\pi})$ increases with $\hat{\pi}$.  Again, this is not surprising.  The higher the prior belief that all channels satisfy $H_1$, the more restrictive one should should be in selecting a channel believed to satisfy that hypothesis.

\section{Appendix: Tables of data}

\begin{table}[h!]
\caption{Optimal thresholds $\pi^*(c)$, for $\hat{\pi} = .5$.}
\begin{tabular}{l | c || c | r}
\hline
$c$ & $\pi^*(c)$ & $c$ & $\pi^*(c)$ \\ \hline
.0025 & .995  & .05 & .865 \\
.005 & .989  & .06 & .840 \\
.01 &  .977  & .07 & .815 \\
.02 & .950   & .08 & .793 \\
.03 & .922 & .09 & .773 \\
.04 & .893 & .1 & .755 \\
\hline
\end{tabular}
\label{fig:table1}
\end{table}

\begin{table}[h!]
\caption{Optimal thresholds $\pi^*(c)$, for $\hat{\pi} = .75$.}
\begin{tabular}{l | c || c | r}
\hline
$c$ & $\pi^*(c)$ & $c$ & $\pi^*(c)$ \\ \hline
.0025 & .995  & .05 & .911 \\
.005 & .990  & .06 & .899 \\
.01 &  .979  & .07 & .889 \\
.02 & .959   & .08 & .879 \\
.03 & .941 & .09 & .871 \\
.04 & .925 & .1 & .864 \\
\hline
\end{tabular}
\label{table2}
\end{table}

\begin{table}[h!]
\caption{Optimal thresholds $\pi^*(\hat{\pi})$, for $c = .01$}
\begin{tabular}{l | c || c | r}
\hline
$\hat{\pi}$ & $\pi^*(\hat{\pi})$ & $\hat{\pi}$ & $\pi^*(\hat{\pi})$ \\ \hline
.025 & .704 & .5 & .977 \\
.05 & .951 & .6 & .978 \\
.1 & .968 & .7 & .979 \\
.2 & .973 & .8 & .980 \\
.3 & .975 & .9 & .982 \\
.4 & .976 & .95 & .985 \\
\hline
\end{tabular}
\label{table3}
\end{table}

\begin{table}[h!]
\caption{Optimal thresholds $\pi^*(\hat{\pi})$, for $c = .03$}
\begin{tabular}{l | c || c | r}
\hline
$\hat{\pi}$ & $\pi^*(\hat{\pi})$ & $\hat{\pi}$ & $\pi^*(\hat{\pi})$ \\ \hline
.025 & .041 & .5 & .922 \\
.05 & .164 & .6 & .930 \\
.1 & .690 & .7 & .937 \\
.2 & .867 & .8 & .946 \\
.3 & .898 & .9 & .960 \\
.4 & .913 & .95 & .972 \\
\hline
\end{tabular}
\label{table4}
\end{table}

\pagebreak

\bibliographystyle{plain}
\bibliography{seqbib}

\end{document}